\documentclass[12pt]{article}

\usepackage[bookmarks,bookmarksnumbered,%
    citebordercolor={0.8 0.8 0.8},filebordercolor={0.8 0.8 0.8},%
    linkbordercolor={0.8 0.8 0.8},%
    urlbordercolor={0.8 0.8 0.8},%
    pagebackref,linktocpage%
    ]{hyperref}

\usepackage[margin=1.0in]{geometry}
\usepackage{graphicx}              
\usepackage{amsmath}               
\usepackage{amsfonts}              
\usepackage{amsthm}                
\usepackage{url}
\usepackage{amssymb}
\usepackage{latexsym}
\usepackage{bm}
\usepackage{enumitem}

\allowdisplaybreaks

\newtheorem{theorem}{Theorem}[section]

\newtheorem{prop}[theorem]{Proposition}
\newtheorem{lemma}[theorem]{Lemma}

\newtheorem{corr}[theorem]{Corollary}

\newtheorem{remark}{Remark}[section]

\theoremstyle{definition}

\let\oldmarginpar\marginpar
\renewcommand\marginpar[1]{\-\oldmarginpar[\raggedleft\footnotesize #1]%
{\raggedright\footnotesize #1}}

\def\keywords#1{\bigskip \par\noindent{\it Keywords and phrases: }#1\par}
\def\AMS#1{\par\noindent{\it 2010 Mathematics Subject Classification: }#1\par}

\DeclareMathOperator{\R}{\mathbb{R}}
\DeclareMathOperator{\C}{\mathbb{C}}
\DeclareMathOperator{\T}{\mathbb{T}}
\DeclareMathOperator{\N}{\mathbb{N}}
\DeclareMathOperator{\Z}{\mathbb{Z}}

\DeclareMathOperator{\st}{\text{ such that }}

\DeclareMathOperator{\ra}{\rightarrow}
\DeclareMathOperator{\half}{\frac{1}{2}}

\DeclareMathOperator{\hyphen}{-}
\DeclareMathOperator{\sgn}{sgn}

\newcommand{\tri}{\Delta}
\newcommand{\tn}[2]{T[#1,#2]}
\newcommand{\tcal}[2]{\Delta(#1 , #2 )}

\newcommand{\Ss}{\section}

\title{Gowers norms for singular measures}
\author{Marc Carnovale}
\date{\today}
\begin{document}

\maketitle


\begin{abstract}
Gowers introduced the notion of uniformity norm $\|f\|_{U^k(G)}$ of a bounded function $f:G\ra\R$ on an abelian group $G$ in order to provide a Fourier-theoretic proof of
Szemeredi's Theorem, that is, that a subset of the integers of positive upper density contains arbitrarily long arithmetic progressions. Since then, Gowers norms have found a number of
other uses,
both within and outside of Additive Combinatorics. The $U^k$ norm is defined in terms of an operator $\tri^k : L^{\infty}(G)\mapsto L^{\infty} (G^{k+1})$. 
In this paper, we introduce an
analogue of the object $\tri^k f$ when $f$ is a singular measure on the torus $\mathbb{T}^d$, and similarly an object $\|\mu\|_{U^k}$.
We provide criteria for $\tri^k \mu$ to exist, which turns out 
to be equivalent to finiteness of $\||\mu|\|_{U^k}$, and show that when $\mu$ is absolutely continuous with density $f$,
then the objects which we have introduced are reduced to the standard $\tri^k f$ and $\|f\|_{U^k(\T)}$. We further introduce a higher-order inner product between measures of finite $U^k$ norm and 
prove a Gowers-Cauchy-Schwarz inequality for this inner product.

\end{abstract}

\tableofcontents

\section*{Acknowledgments}

Many thanks to Prof.'s Izabella Laba and Malabika Pramanik for introducing me to the question of arithmetic progressions in fractional sets which motivated this work, 
for their patience and many suggestions with innumerable different drafts and pre-drafts,
for the sharing of their expertise in Harmonic Analysis, for funding my master's degree, and much more. 

Thanks to Nishant Chandgotia for his forgiveness of the mess I made of our office for two years and his bountiful friendship.

Thanks to Ed Kroc, Vince Chan, and Kyle Hambrook for the frequent use of their time and ears and their ubiquitous encouragement.


\noindent \keywords{Gowers norms, Uniformity norms, singular measures, Finite point configurations, Salem sets, Hausdorff dimension, Fourier dimension}
\vskip0.2in

\noindent \AMS{28A78, 42A32, 42A38, 42A45, 11B25, 42B35, 42B10}


\section{Introduction}

In 2001, Gowers developed a new proof of Szemeredi's Theorem that every dense enough subset of the integers contains arbitrarily long arithmetic progressions $a,a+b,\cdots,a+(k-1)b$, see \cite{gowers}. 
His method revolved around the introduction of uniformity norms $\|\cdot\|_{U^k(\Z_N)}$, which measure the extent to which a bounded function on $\Z_N$ is ($k+1$-st degree polynomially)
``structured''. Since then, uniformity norms have found applications in diverse topics, notably 
progressions in primes \cite{gtao}, probabilistically checkable proofs \cite{samorodnitsky}, multi-linear oscillatory integrals \cite{christ-tao}, 
the bi-linear Hilbert transform along curves \cite{li}, boundedness of paraproducts \cite{kovac}, and others. 

Gower's original definition of the $U^k(\Z_N)$ norms proceeded as follows: for a bounded function $f:\Z_N\ra\R$ and a $\bm{u}=(u_1,\dots,u_{k+1})\in\Z_N^{k+1}$, inductively define
\begin{align*}
 \tri^1_{u_1}f(x) &= f(x)f(x-u_1)\\
\tri^{k+1}_{\bm{u}}f(x)&= \tri_{u_{k+1}}^1 \tri^k_{\bm{u}'}f(x)
\end{align*}
where $\bm{u}'=(u_1,\dots,u_k)$. 

Then the $k$-th order uniformity norm of $f$ is given by

\begin{align*}
 \|f\|_{U^k(\Z_n)} = \left( \sum_{x\in\Z_N,\bm{u}\in\Z_N^k} \tri^k_{\bm{u}}f(x)\right)^{\frac{1}{2^k}}
\end{align*}

In this paper, we extend the domains of definition of $\tri^k$ and $\|\cdot\|_{U^k}$ to the class of positive finite singular Radon measures on $\T^d$.
For a measure $\mu$ on $\T^d$ we construct a measure $d\tri^k\mu(x;\bm{u})$ on 
$\T^{d(k+1)}$ and  provide a definition for $\|\mu\|_{U^k}$ which (we show but cannot at first assume) reduces to $\left( \tri^k\mu(\T^{d(k+1)})\right)^{\frac{1}{2^{k}}}$. 

Let us say that $\mu\in U^{k+1}$ if the finite measure $\tri^k\mu$ exists on $\T^{d(k+1)}$ and $\|\mu\|_{U^{k+1}}<\infty$. Then our main result can be summarized as the 
assertion that $\tri^{k+1}\mu$ exists if $|\mu|\in U^{k+1}$ 
(Theorem \ref{thm:FirstTheorem}).

Our motivation for this work stems from potential applications in Geometric Measure Theory, and particularly from the paper \cite{Laba} in which Laba and Pramanik demonstrate that a measure supported in $[0,1]$ with
Fourier dimension close enough to $1$ contains in its support three-term arithmetic progressions $a,a+b,a+2b$. Here, the Fourier dimension of a measure $\mu$ is defined as 
\begin{align*}
 \dim_{\mathbb{F}}\mu : = \sup \{\beta\in[0,1] : \exists C \st |\widehat{\mu}(\xi)|\leq C(1+|\xi|)^{-\frac{\beta}{2}}\}
\end{align*}

In following papers,  we use the machinery of uniformity norms to demonstrate that for a given $k\in\N$, 
measures satisfying an appropriate generalization of the above Fourier dimension assumption contain $k$-term arithmetic progressions.

{
\Ss{The intersection operator}\label{ch:general}

Suppose $\nu_0$, $\nu_1$ are complex Radon measures on $\T^d\times\T^r$. Let $(\phi_n)_{n\in\N}$ be an approximate identity on $\T^d\times\T^r$.

Define

\[ \tn{\nu_0}{\nu_1} := \lim_{n\ra\infty} \int \phi_n\ast\overline{\nu_1}(x-u_{r+1};u')\,d\nu_0(x;u')\,du_{r+1}.\]

Define the projection operator $P$ by
\[\int f(u)\,dP\nu(u) := \int f(u)\,d\nu(x;u).\]

\begin{lemma}\label{thm:fouriernorm} Let $\nu$ be a complex Radon measure on $\T^d\times\T^r$. Then
 \[ \tn{\nu}{\nu} = \sum_{\eta\in\Z^{r}} |\widehat{\nu}(0;\eta)|^2 = \|P\nu\|_{L^2}^2.\]
\end{lemma}
\begin{proof}
 Fix $n\in\N$. Then 
 \[\int \phi_n\ast\overline{\nu}(x-u_{r+1};u')\,d\nu(x;u')\,du_{r+1} = \lim_{m\ra\infty} \int \phi_n\ast\overline{\nu}(x-u_{r+1};u')\phi_m\ast\nu(x;u')\,dx\,du.\]
 
 Applying Plancherel's Identity, this becomes
 \[\sum_{\eta\in\Z^r} \widehat{\phi_n}(-\eta)\widehat{\phi_m}(\eta) \widehat{\nu}(0;\eta)\overline{\widehat{\nu}(0;\eta)} = \sum_{\eta\in\Z^r} \widehat{\phi_n}(-\eta)\widehat{\phi_m}(\eta) |\widehat{\nu}(0;\eta)|^2.\]
If $\widehat{\nu}(0;\eta)\in\ell^2$, then
\begin{align}\label{sharot}\lim_{n\ra\infty}\lim_{m\ra\infty}\sum_{\eta\in\Z^r} \widehat{\phi_n}(-\eta)\widehat{\phi_m}(\eta) |\widehat{\nu}(0;\eta)|^2 = \sum_{\eta\in\Z^r} |\widehat{\nu}(0;\eta)|^2\end{align}
by Dominated Convergence since $|\widehat{\phi_n}|\leq 1$, in which case we have proven the theorem. 

Suppose $\widehat{\nu}(0;\eta)\notin\ell^2$, so that we must show that the left-hand side of (\ref{sharot}) is infinite.

In this case the measure $P\nu  = \int f(u)\,d\nu(x;u)$, whose Fourier transform is $\widehat{P\nu}(\eta) = \widehat{\nu}(0;\eta)$
is not in $L^2$.

Let $P\nu_s$ denote the singular part of $P\nu$. Let $\psi_n(u) = \int\phi(x,u)\,dx$.  Since $\psi_n\ast P\nu$ diverges to infinity $P\nu_s$-a.e,
we surmise that if the singular part $P\nu_s$ of $P\nu$ is non-trivial, then the left-hand side of (\ref{sharot}) diverges to $\infty$, as we sought to show.

In the case that $P\nu$ is absolutely continuous (say with density equal to the function $g$), we may set the integrand 
$\psi_n\ast g(u)\,g(u)=: f_n(u)$.
Then we know that $f_n\ra |g|^2$ at Lebesgue-almost every point, and since$g\notin L^2$, applying Egorov's Theorem we see that $\int f_n\ra\infty$
which completes the proof.

\end{proof}

Define the Radon measure $\tcal{\nu}{\nu}$ by

\[ \int f(x;y;v) \,d\tcal{\nu}{\nu}(x;y;v) := \lim_{n\ra\infty} \phi_n\ast\overline{\nu_1(x-v;y)}\,d\nu_0(x;y)\,dv.\]

Notice that $\tcal{\nu}{\nu}$, if it exists, is a measure on $\T^d\times\T^r\times\T^d$. We will often replace $y$ by $u'$, $v$ by $u_{r+1}$ and write $u=(u',u_{r+1})$ 
in the above definition.

\begin{lemma}\label{thm:mon}
For any $\xi\in\Z^d$,
 \[\sum_{\eta\in\Z^r}|\widehat{\nu}(\xi;\eta)|^2\leq \sum_{\eta\in\Z^r}|\widehat{|\nu|}(0;\eta)|^2 = \tn{|\nu|}{|\nu|}.\]
\end{lemma}
\begin{proof}
 The second equality follows by Lemma \ref{thm:fouriernorm}.
 
 Set $f_{n,\xi}(u) = \int e^{-2\pi i \xi\cdot x}\phi_n\ast\nu(x;u)\,dx$. Let $\phi_n$ be positive. Note that $\hat{f_{n,\xi}}(\eta) = \widehat{\phi_n}(\xi;\eta)\widehat{\nu}(\xi;\eta)$. By Plancherel and the 
 triangle inequality we have
 \[ \sum_{\eta}|\widehat{\phi_n}(\xi;\eta)\widehat{\nu}(\xi;\eta)|=\int \left|f_{n,\xi}(u)\right|^2\,du \leq\int \left|\int\phi_n\ast|\nu|(x;u)\,dx\right|^2\,du = \sum_{\eta}|\widehat{\phi_n}(\xi;\eta)\widehat{|\nu|}(0;\eta)|^2
  . \]
\end{proof}

\begin{prop}\label{thm:fouriertransform}
 Suppose $\tn{|\nu_0|}{|\nu_0|},\tn{|\nu_1|}{|\nu_1|} < \infty$. Then for any $(\xi;\eta)\in \Z^d\times\Z^{r+d}$
 \begin{align}\label{fouriertransform}\sum_{c\in\Z^r} \widehat{\nu_0}(\xi+\eta_{r+d};c)\overline{\widehat{\nu_1}(\eta_{r+d},c-\eta')}\end{align}
  is uniformly absolutely summable. Furthermore, $\tcal{\nu_0}{\nu_1}$ exists and $\widehat{\tcal{\nu_0}{\nu_1}}(\xi;\eta) = (\ref{fouriertransform})$.
 \end{prop}
 \begin{proof}
  The summability of (\ref{fouriertransform}) follows from Cauchy-Schwarz and Lemma \ref{thm:mon}. So Dominated Convergence gives
   \begin{align}\label{lineabove}\lim_{n\ra\infty}\sum_{c\in\Z^r} \widehat{\phi_n}(-\eta_r;\eta-c)\widehat{\nu_0}(\xi+\eta_{r+d};c)\overline{\widehat{\nu_1}(\eta_r,c-\eta')} 
    = \sum_{c\in\Z^r} \widehat{\nu_0}(\xi+\eta_{r+d};c)\overline{\widehat{\nu_1}(\eta_r,c-\eta')}.\end{align}

    Using Plancherel, one computes that
    \[ \int e^{-2\pi i (\xi\cdot x + \eta\cdot u} \phi_n\ast\overline{\nu_1}(x-u_{r+1};u')\,d\nu_0(x;u')\,du_{r+1} = \sum_{c\in\Z^r} \widehat{\phi_n}(-\eta_r;\eta-c)\widehat{\nu_0}(\xi+\eta_{r+d};c)\overline{\widehat{\nu_1}(\eta_r,c-\eta')}.\]
    
   Since the unit ball in the space of Radon measures is weak$^*$-compact by the Banach-Alaoglu theorem, it is enough to check that a putative weak$^*$ limit $\lim_{n\ra\infty}\rho_n$
   exists on a dense subclass of $C(\T^d\times\T^r\T^d)$ and that the $\rho_n$ are of uniformly bounded mass. We have the existence of the limit on the dense subclass of trigonometric
   polynomials as a consequence of (\ref{lineabove}) and linearity, and to demonstrate uniformly bounded mass we have
   \begin{align*}&|\int f(x,u) \phi_n\ast\overline{\nu_1}(x-u_{r+1};u')\,d\nu_0(x;u')\,du_{r+1}|\\\leq&
   \|f\|_{L^{\infty}}\int  \phi_n\ast\overline{|\nu_1|}(x-u_{r+1};u')\,d||(x;u')\,du_{r+1}\leq \tn{\nu_0}{\nu_0}\tn{\nu_1}{\nu_1}.
   \end{align*}
where in the first inequality we have applied the triangle inequality, and to arrive at the second
we have used Plancherel's identity, the bound $|\phi_n|\leq 1$, the Cauchy-Schwarz inequality, and Lemma \ref{thm:fouriernorm}.    

Thus we obtain existence of $\tcal{\nu_0}{\nu_1}$ together with the statement about its Fourier transform.
 \end{proof}

 \begin{prop}\label{thm:absolutevalue}
  Suppose that $\tcal{\nu}{\nu}$ is a finite Radon measure. Then $\tcal{|\nu|}{|\nu|}$ exists and
  \[|\tcal{\nu}{\nu}| = \tcal{|\nu|}{|\nu|}.\]
  
  Further
  \[\tn{|\nu|}{|\nu|} = \|\tcal{\nu}{\nu}\|\]
  where $\|\cdot\|$ denotes the total variation norm.
 \end{prop}
 \begin{proof}

Suppose that the finite measure $\tcal{\nu}{\nu}$ exists. We complete the proof in several stages.

Recall that for a measure $\rho$ on $\T^d\times\T^{r}$, $P$ denotes the projection onto the measures on $\T^{r}$.

Step {1}. 
Let $\sigma_n\rightarrow\sigma=\sgn(\nu)$ be a sequence of trigonometric polynomials. Then

\begin{align*}
 &P(\sigma_n\nu)\in L^2
 \\\text{ with } &\|P(\sigma_n\nu)\|_2\leq \tcal{\nu}{\nu}
\end{align*}

\begin{proof}[Proof of Step {1}.] Suppose first that $\sigma_m(x;{u}') = e^{-2\pi i (\kappa_0\cdot x+ {\kappa}\cdot{u}')}=e_{\kappa_0}(x)e_{{\kappa}}({u}')$ 
for some $(\kappa_0;{\kappa})\in\T^d\times\T^{r}$.

By a modification of the argument in Lemma \ref{thm:fouriernorm}, we have
\begin{align}\label{resd}
 &\sum_{{\eta}\in\Z^{r}} |\widehat{\sigma_m\nu}(0;{\eta})|^2 \\=&
 \lim_{n\ra\infty} \sum_{{\eta}\in\Z^{r}} \widehat{\Phi_n e_{(\kappa_0,{\kappa})}}(0;\eta)|\widehat{\sigma_m\nu}(0;{\eta})|^2\notag{}
 \end{align}
 
 Using that $\sigma_m=e_{(\kappa_0,{\kappa})}$, we then write
 \begin{align*}(\ref{resd})=&
 \lim_{n\ra\infty} \sum_{{\eta}\in\Z^{r}} \widehat{\Phi_n}(\kappa_0;{\eta}+{\kappa})\widehat{\nu}(\kappa_0;{\eta+\kappa})
 \overline{\widehat{\nu}(\kappa_0;{\eta+\kappa})}\\=&
 \lim_{n\ra\infty} \sum_{{\eta}\in\Z^{r}} \widehat{\nu}(\kappa_0;{\eta+\kappa})
 \overline{\widehat{\Phi_n\ast\nu}(\kappa_0;{\eta+\kappa})}\\=&
 \lim_{n\ra\infty} \int e_{\kappa_0}(x-u_{r+1})\Phi_n\ast\overline{\nu}(x-u_{r+1};{u}')\,d\nu(x;{u}')\,du_{r+1}
\end{align*}
 where the last line above follows from Plancherel and the Dominated Convergence theorem in light of Lemma \ref{thm:mon}. 
 
 Using the definition of the measure $\tcal{\nu}{\nu}$ and that $e_{(\kappa_0,{\kappa})}=\sigma_m$, this means that we in fact have
 \begin{align*}(\ref{resd})=&\int e_{(\kappa_0;0)}(x-u_{j+1};\bm{u}') \,d\tcal{\nu}{\nu}(x;\bm{u})=\int \tilde{\sigma_m}(x-u_{j+1}) \,d\tcal{\nu}{\nu}(x;\bm{u})
\end{align*}
where $\tilde{\sigma_m}(x) = \sigma_m(x;0)$.

By linearity, the same equality then holds when $\sigma_m$ is instead any trigonometric polynomial, and the triangle inquality gives $(\ref{resd})\leq\|\sigma_m\|_{\infty}
\|\tcal{\nu}{\nu}\|$
. Taking limits then yields the claim since $\|\sigma\|_{L^{\infty}} = 1$.

\end{proof}

Step {2}. \begin{align*}P(\sigma_n\nu)\stackrel{L^2}{\ra} P(\sigma\nu)\end{align*}

\begin{proof}[Proof of Step {2}.]
Since $\sigma_m\nu\ra\sigma\nu$ weak$^*$, we have weak$^*$ convergence of $P(\sigma_m\nu)$ to $P(\sigma\nu)$. 
Since we have by Step {1} that $\|P(\sigma_m\nu)\|_2 \leq \|\tcal{\nu}{\nu}\|$, we have (by weak compactness of the unit ball) that any subsequence of $P(\sigma_m\nu)$
has a subsubsequence which converges weakly in $L^2$. So necessarily the whole sequence converges to $P(\sigma\nu)\in L^2$ with 

\begin{align*}\|P(\sigma\nu)\|_2 \leq \|\tcal{\nu}{\nu}\|\end{align*}

\end{proof}

Step {3}. Therefore $\tcal{|\nu|}{|\nu|}(\T^d\times\T^r\times\T^d) < \infty$.

\begin{proof}[Proof of Step {3}.]
By Step {2}, we have $\tn{|\nu|}{|\nu|} <\infty$, which means by Lemma \ref{thm:fouriernorm} that 
\[\tcal{|\nu|}{|\nu|}(\T^{d+r+d)}) < \infty.\]

\end{proof}

\vspace{14pt}
We will now be done once we close the induction introduced in Step {1} by showing that $\left |\tcal{\nu}{\nu}\right| = \tcal{|\nu|}{|\nu|}$. Continuing
\vspace{14pt}

Step {4}. If $\tn{\nu}{\nu}<\infty$, one may write $d\nu(x;{u}') = d\mu_{{u}'}(x)\,d{u}'$ for some finite Radon measures $\mu_{{u}'}$,
and for any trigonometric polyomials $f_1,f_2\in C(\T^{d(r+1)})$
\begin{align*}
\int f_1(x;{u}')\bar{f_2}(x-u_{r+1};{u}') d\tcal{\nu}{\nu}(x;{u}) = \int \int f_1(x;{u}')\,d\mu_{{u}'}(x) \int \bar{f_2}(y;{u}')d\mu_{{u}'}(y)\,d{u}'
\end{align*}

\begin{proof}[Proof of Step {4}.]
Since $P\nu\in L^2$ by Lemma \ref{thm:fouriernorm}, we have that $ d\nu(x;{u}') = d\mu_{{u}'}(x) \,d{u}'$ for some measures $\mu_{{u}'}$ defined 
Lebesgue for almost every ${u}'$. 

We can calculate that 
\begin{align*}
&\int \int f_1(x;{u}')\,d\mu_{{u}'}\int \bar{f_2}(y;{u}')\,d\mu_{{u}'}(y)\,d{u}'\\=&
\sum_{{\eta}} \widehat{f_1\mu_{\cdot}}(0;{\eta})\overline{\widehat{f_2\mu_{\cdot}}(0;{\eta})}
\end{align*}
which is the same as
\begin{align*}
&\sum_{{\eta}} \widehat{(f_1\nu)}(0;{\eta})\overline{\widehat{({f_2}\nu)}(0;{\eta})}
\end{align*}

For $f_1,f_2$ complex exponentials, this is easily seen to equal  
\[\int f_1(x-u_{r+1};{u}') \bar{f_2}(x;{u}')\,d\tcal{\nu}{\nu}(x;{u}),\] and we have this equality also in the case
that the $f_i$ are trigonometric polynomials.

\end{proof}

Step {5}. For any trigonometric polyomials $f_1,f_2\in C(\T^{d+r+d)})$

\begin{align}\label{coroon}
& \int \left(f_1\sigma\right)(x;{u}')\left(f_2\sigma\right)(x-u_{r+1};{u}') d\tcal{\nu}{\nu}(x;{u}) \\=& \int f_1(x;{u}')f_2(x-u_{r+1};{u}') d\tcal{|\nu|}{|\nu|}(x;{u})
\end{align}

\begin{proof}[Proof of Step {5}.] 

Since $\tcal{|\nu|}{|\nu|}(\T^{d(r+1)}) < \infty$ by Step {3}, we may apply the conclusion of Step {4} to the measure $\nu$. One has that 
\[d|\nu|(x;u') = \sigma(x;u') d\nu(x;u') = \sigma(x;u') \,\mu_{u'}(x)\,du'\]
so that necessarily $d|\mu_{u'}(x)| = \sigma(x;u')\,d\mu_{u'}(x)$. Two applications of Step {4} give
\begin{align}\label{bestday}
 &\int f_1(x-u_{r+1};{u}')f_2(x;{u}')\,d\tcal{|\nu|}{|\nu|}(x;{u}) = \int \int f_1(x;{u}')\,d\left|\mu\right|_{{u}'}(x)\int f_2(y;{u}')\,d\left|\mu\right|_{{u}'}(y)\,d{u}'\\=&
 \int [\sigma f_1](x-u_{r+1};{u}')[\sigma f_2] (x;{u}')\,d\tcal{\nu}{\nu}(x;{u})
\end{align}

Step {6}. Thus 
\begin{align*}
 \left|\tcal{\nu}{\nu}\right| = \tcal{|\nu|}{|\nu|}
\end{align*}
\end{proof}

\begin{proof}[Proof of Step {6}.] Since the trigonometric polynomials are dense in the space of continuous functions, we see that in fact (\ref{coroon}) holds for all continuous functions
$f$, which  (since $\tcal{|\nu|}{|\nu|}$ is a positive measure) means that 
the sign of the measure $\tcal{\nu}{\nu}$ is $\sigma(x;{u}')\sigma(x-u_{r+1};{u}')$. That $|\tcal{\nu}{\nu}|= \tcal{|\nu|}{|\nu|}$ then follows from Step {5}.

\end{proof}
\end{proof}

 \begin{corr}
  $\tcal{\nu}{\nu}$ exists if and only if $\tn{|\nu|}{|\nu|}<\infty$.
 \end{corr}
\begin{proof}
  The left implication is Proposition \ref{thm:fouriertransform} while the right is Proposition \ref{thm:absolutevalue}.
\end{proof}

 \Ss{Gowers norms for singular measures}
 Set $\tri^0\mu :=\mu$. Define
 \[ \tri^{k+1}\mu := \tcal{\tri^k\mu}{\tri^k\mu}.\]
 
 We establish a convention on indexing. For elements  $\iota=(\iota_1,\dots,\iota_{k+1})\in\{0,1\}^{j}$ and $\kappa\in\{0,1\}^{j'}$
 concatenate $\kappa\iota = (\iota_1,\dots,\iota_j,\kappa_1,\dots,\kappa_{j'})$. In particular, $\iota=\iota_{j}\iota_{j-1}\cdots\iota_1$. Write $\iota' := (\iota_1,\dots,\iota_{j-1})$.
 Further, let $\iota_{>j'}:=(\iota_{j'+1},\iota_{j'+2},\dots,\iota_{j})$	 and similarly for $\iota_{< j'}$, $\iota_{\geq j}'$ and $\iota_{\leq j'}$. Thus $\iota_{>0} = \iota$.
 
 For $\iota\in\{0,1\}^{k+1}$ and $\mu_{\iota}$ measures on $\T^d$, adopt the notation $\bm{\mu} = (\mu_{\iota})_{\iota\in\{0,1\}^{k+1}}$, and for any $\kappa\in\{0,1\}^{k-j}$,
 $\bm{\mu_{\kappa}} := (\mu_{\kappa\iota_{<j}})_{\iota_{<j}\in\{0,1\}^{j-1}}$. 
 
 Define
 \[ \tri^{k+1}(\bm{\mu}) = \tcal{\tri^k(\bm{\mu_0})}{\tri^k(\bm{\mu_1})}.\]
 
 Thus if $\mu_{\iota} = \mu$ for all $\iota$, then $\tri^{k+1}(\bm{\mu}) = \tri^{k+1}\mu$.
 
 Define 
 \[ \|\mu\|_{U^{k+1}} :=  \left(\tn{\tri^k\mu}{\tri^k\mu}\right)^{\frac{1}{2^{k+1}}}\]
 if $\tri^k\mu$ exists and to be $\infty$ otherwise.

 Define
 \[U^{k+1} = \{ \mu : \|\mu\|_{U^{k+1}} < \infty\}.\]
 
 Finally, define 
 \[\langle\bm{\mu}\rangle := \langle\bm{\mu_0},\bm{\mu_1}\rangle := \tn{\tri^k(\bm{\mu_0})}{\tri^k(\bm{\mu_1})}.\]
 
 Specializing the results of Section \ref{ch:general} to the measure $\nu = \tri^k\mu$, we obtain
 \begin{prop}\label{thm:gowersnorms}
  Let $k\in\N$. The $U^{k+1}$ norm $\|\mu\|_{U^{k+1}}$ is well-defined (with a value in $[0,\infty]$) for all finite complex Radon measures $\mu$ on $\T^d$. 
  The measure $\tri^{k+1}({\mu})$ exists  if and only if $\||\mu|\|_{U^{k+1}}<\infty$. For all $(\xi;\eta)\in\Z^d\times\Z^{kd}$
  \[ \widehat{\tri^{k+1}(\bm{\mu})}(\xi;\eta) = \sum_{c\in\Z^{kd}} \widehat{\tri^k(\bm{\mu_0})}(\xi+\eta_{k+1};c)\overline{\widehat{\tri^k(\bm{\mu_0})}(\eta_{k+1};c-\eta'})\]
  which is uniformly absolutely summable if $\|\mu\|_{U^{k+1}}<\infty$.
  
  In particular, 
  \[ \|\mu\|_{U^{k+1}} = \left(\sum_{c\in\Z^{kd}} |\widehat{\tri^k\mu}(0;c)|^2\right)^{\frac{1}{2^{k+1}}}.\]
 \end{prop}

\begin{corr}\label{thm:factories}
 Let $\Psi_{n} : \T^{d(k+1)}\ra\T^d$ be an approximate identity and $\mu_{\iota}\in U^{k+1}$, $\iota\in\{0,1\}^{k+1}$ . Then for any $f\in C(\T^{d(k+2)})$,

\begin{align}\notag{}
&\int f \,d\tri^{k+1}(\bm{\mu})(x;\bm{u}) = \\\label{left}
 &\lim_{n\ra\infty} \int f(x,u_1,\dots,u_{k+1}) \overline{\Psi_{n}\ast\tri^k(\bm{\mu_1})(x-u_{k+1};\bm{u}')}\Psi_{n}\ast\tri^k(\bm{\mu_0})(x;\bm{u}') \, dx d\bm{u} 
\end{align}

\end{corr}
\begin{proof}
One needs only compare Fourier transform of $\tri^{k+1}(\bm{\mu})$ and the limit as $n\ra\infty$ of the Fourier transform of the function $
\overline{\Psi_{n}\ast\tri^k(\bm{\mu_1})(x-u_{k+1};\bm{u}')}\Psi_{n}\ast\tri^k(\bm{\mu_0})(x;\bm{u}')$ to see the equality.
%
%
%
%
\end{proof}

\Ss{The uniformity norm of an absolutely continuous measure}\label{ch:uniformity}
Recall that given any compact abelian group $G$ with Haar measure $dx$, the $k$-th order uniformity norm $U^k$ of a function $f$ on $G$ is given by

\begin{align}
 \|f\|_{U^k(G)}^{2^k} := \int_{G\times G^k} \tri^k f(x;\bm{u})\,dx\,d\bm{u}
\end{align}
where $\tri^0f(x) :=f(x)$ and inductively, $\tri^{k+1}f(x;\bm{u}) := \tri^kf(x;\bm{u})\overline{\tri^k f(x-u_{k+1};\bm{u}')}$.

In this section, we show that if $f$ is a function on $\T^d$ with finite $U^{k+1}(\T^d)$ norm, then the uniformity norm $\|f\|_{U^{k+1}(\T^d)}$ coincides with $\|fdx\|_{U^{k+1}}$ as defined in previous 
sections, and that the measure $\tri^{k+1}(f\,dx)$ has a density given by $\tri^{k+1} f(x;\bm{u})$. In more detail, our main result is the following.

\begin{lemma}\label{thm:uniformity}
 Suppose that the positive function $f:\T^d\ra\R$ has a finite $U^{k+1}$ norm for some $k$. Then the measure $d\mu=fdx$ satisfies
 \begin{align*}
  \|\mu\|_{U^{k+1}} = \|f\|_{U^{k+1}(\T^d)}
 \end{align*}
and the finite measure
 $\tri^{k+1}\mu$ exists, is absolutely continuous with respect to the Lebesgue measure on $\T^{d(k+2)}$, and has a density given by $\tri^{k+1} f$.
\end{lemma}
\begin{remark}
 The point of this Lemma is that we are assuming no $L^p$ regularity of $f$, only finiteness of the $U^{k+1}$ norm.
\end{remark}

\begin{proof}
 We induct on $k$. That $\tri^0 (fdx) = (\tri^ 0 f) dx$ is a tautology. Fix $k\geq 0$ and let us state the inductive assumption

\begin{center}
\begin{enumerate}[label={(\Roman*)}]
 \item $d \tri^k (f\,dx) (x;\bm{u}') = \tri^k f(x;\bm{u}') \,dx\,d\bm{u}'$\label{itm:I}
\end{enumerate}
\end{center}

Let $g(u') = \int \tri^kf(x;u')\,dx$. A change of variables shows that the finiteness of $\|f\|_{U^{k+1}(\T^d)}$ is equivalent to the statement that $g\in L^2$. An argument like the 
one in Lemma \ref{thm:mon} shows that for any $\xi\in\Z^d$, $g_{\xi}(u'):=\int \exp(\xi\cdot x) \tri^kf(x;u')\,dx\in L^2$. Further, 
\[ \int \left|\int \exp(\xi\cdot x)\left( \phi_n\ast\tri^kf(x;u')-\tri^kf(x;u')\right)\,dx\right|^2\,du'= \lim_{n\ra\infty} \sum_{\eta\in\Z^{kd}} |1-\widehat{\phi_n}(0;\eta)|^2 |\widehat{\tri^k f}(\xi;\eta)|^2.\]
Since $\sum_{\eta} |\tri^k f(\xi;\eta)|^2 = \|g_{\xi}\|_{L^2}^2$, by Dominated Convergence
the above converges to $0$, showing that 
\[\int \exp(\xi\cdot x)\phi_n\ast\tri^kf(x;u')\,dx \stackrel{L^2}{\longrightarrow} \int \exp(\xi\cdot x)\tri^kf(x;u')\,dx = g_{\xi}(u').\]

For any trigonometric polynomial $p(x;u) = p_0(x;u')p_1(x-u_{k+1};u')$ with $p_1(x;u') = e_{\xi}(x) e_{\eta}(u')$, we have
\begin{align*}
&\int p\, d\tcal{\tri^k f\,dx}{\tri^k f\,dx}\\=&\lim_{n\ra\infty} \int e_{\eta}(u')p_1(x;u')\left( \int e_{\xi}(y)\phi_n\ast\tri^kf(y;u')\,dy\right) \tri^k f(x;u')\,dx\,du \\=&
\int e_{\eta}(u)g_{\xi}(u')p_1(x;u') \tri^k f(x;u')\,dx\,du = \int p(x;u) \tri^kf(x-u_{k+1};u')\tri^k f(x;u')\,dx\,du
\end{align*}
where the first equality follows by the inductive hypothesis \ref{itm:I} and the definition of $\tcal{}{}$, the second equality via the $L^2$ convergence of the integral to $g_{\xi}$, 
and the third equality from the definition of $g_{\xi}$. This shows that $\tri^{k+1}(f\,dx)(x;u) = \tri^{k+1}f(x;u)\,dx\,du$, since the equality holds for trigonometric polynomials
by linearity, and hence for all continuous functions. Thus by induction, \ref{itm:I} holds for all $k$ such that $\|f\|_{U^{k}}<\infty$

\end{proof}

\Ss{Intertwining mollification and the Gowers-Cauchy-Schwarz inequality}\label{ch:mollification}

The symbol $\phi^{\ast^n}$ will refer to $\phi\ast\cdots\ast\phi$ where $n$ copies of $\phi$ are convolved,
and $\phi \star g(x;u)$ for some multivariate function $g:\T^d\times\T^{dn}\ra\T$ will always refer to the (partial) convolution of $ \phi$ and $g$ with respect to the $x$ variable.
We will also need to handle the slightly more general situation in which a collection $\phi = (\phi^j)_{j=1}^n$ is given, in which case $\phi^{\ast^n} := \phi^1\ast\cdots\ast\phi^n$.

\begin{corr}\label{thm:phin} For  each $\iota\in\{0,1\}^{k+1}$, let $\mu_{\iota}\in U^{k+1}$. Let $\phi_{\iota}=(\phi_{\iota}^n)_{n=1}^{k+1}\in L^1$ be positive functions 
with $\|\phi_{\iota}^n\|_{L^1} =1$.
Then

\begin{align}\label{toshow22}
&\int \prod_{\iota\in\{0,1\}^{k+1}}\mathcal{C}^{|\iota|} \phi_{\iota}^{\ast^{k+1}}\ast\mu_{\iota}(x-\iota\cdot {{u}})\,dx\,d{{u}} \\\leq&
\prod_{\iota'\in\{0,1\}^{k}} \left[\int \overline{\phi_{\iota}^{[k]}\ast\tri^k\mu_{1\iota'}(x-u_{k+1}';{{u}}')}\phi_{\iota}^{[k]}\ast\tri^k\mu_{0\iota'}(x;{{u}}')\,dx\,d{{u}}\right]^{\frac{1}{2^{k+1}}}
\\\leq& \prod_{\iota\in\{0,1\}^{k+1}} \left[\phi_{\iota}^{[{k+1}]}\ast\tri^{k+1}\mu_{\iota}(\T\times\T^{d(k+1)})\right]^{\frac{1}{2^{k+1}}}\left(= \prod_{\iota\in\{0,1\}^{k+1}}\|\mu_{\iota}\|_{U^{k+1}}\right)
\end{align}
\end{corr}

An immediate consequence of Corrollary \ref{thm:phin} is 

\begin{prop}\label{thm:GCS}Suppose for $k\in\N$ that the $2^k$ measures $\mu_{\iota}\in U^{k+1}$. Then

\begin{align}\label{ttoshow2}
|\langle \bm{\mu}\rangle| \leq \prod_{\iota\in\{0,1\}^{k+1}} \|\mu_{\iota}\|_{U^{k+1}}
\end{align}
\end{prop}

In order to obtain either of these results, we need to prove a stronger inequality, Lemma \ref{thm:phin1} below.



 Given $t=(t_0;t_1,\dots,t_k)$, let $\overline{t} = (t_1,\dots,t_k)\in\T^{dk}$. Sometimes $t=(t_0,\dots,t_{n})$ or $t=(t_0,\dots,t_{n+1})$ depending whether $t$ is the argument of $\phi^{[n]}$ or
 of $\phi^{[n+1]}$.

Let $\phi^{[0]} = \phi$. 

For $\phi = (\phi^j)_{j=0}^n$ define

\begin{align}\notag{}
\phi^{[n]}(t) = \int \phi^0(t_0+\sum_{j=1}^{n}c_j) \phi^1(-c_1)\phi^1(-t_1-c_1)\cdots\phi^n(-c_n)\phi^n(-t_n-c_n) \, dc
\end{align}

Notice that 

\begin{align}\label{notice}
&\displaystyle{}\phi^{n+1}\star \phi^{[n]}(t_0,\dots,t_n)= \int \phi^{[n+1]}(t_0,\cdots,t_{n+1}) \, dt_{n+1}
\end{align}


The following is essentially the statement that the Fourier transform of a convolution is the product of the Fourier transforms,
and the proof is precisely the same

\begin{lemma}\label{thm:Fourier mollifier} Let $(\xi;{{\eta}})\in\T^d\times\T^{dn}$. Then 
\begin{align}\label{82ma}
\widehat{\phi^{[n]}}(\xi;{{\eta}}) =\displaystyle{}\widehat{\phi^0}(\xi)\cdot\widehat{\phi^1}(-\eta_1)\cdots\widehat{\phi^n}(-\eta_n)\cdot\widehat{\phi^1}(\xi-\eta_1)\cdots\widehat{\phi^n}(\xi-\eta_n) 
\end{align}
\end{lemma}
\begin{proof}
Expanding $\phi^{[n]}$ according to its definition and integrating first in $t_0$ we have 
\begin{align}\notag{}
\widehat{\phi^{[n]}}(\xi;{{\eta}})=&\int \phi^{[n]}(t)e^{-2\pi i t\cdot (\xi;{{\eta}})}\,dt\\\notag{}=& 
\int \phi^0(t_0 + \sum_{j=1}^n c_j)\left[\prod_{i=1}^n \phi^i(-c_i) \phi^i(-t_i-c_i)\right]\exp(\xi t_0 + \sum_{j=1}^n \eta_jt_j) \, dt_0 d\overline{t} dc\\\label{slowly}=& 
\widehat{\phi^0}(\xi) \int \exp( \sum_{j=1}^n c_j(-\xi)) \left[\prod_{i=1}^n \phi^i(-c_i) \phi^i(-t_i-c_i)\right]\exp(-2\pi i ( \sum_{j=1}^n \eta_jt_j)) \, d\overline{t} dc
\end{align}
then integrating in the remaining $t_i$
\begin{align}\notag{}
&(\ref{slowly})=\\\label{colr}&
\widehat{\phi^0}(\xi)\widehat{\phi^1}(-\eta_1)\cdots\widehat{\phi^n}(-\eta_n) \int \exp( \sum_{j=1}^n c_j(-\xi)) \left[\prod_{i=1}^n \phi^i(-c_i)\right]\exp(\eta_1 c_1)\cdots\exp(\eta_nc_n) \,dc
\end{align}
Collecting terms, this is
\begin{align}\notag{}
(\ref{colr})=&
\widehat{\phi^0}(\xi)\widehat{\phi^1}(-\eta_1)\cdots\widehat{\phi^n}(-\eta_n) \int \exp( \sum_{j=1}^n c_j(\eta_j-\xi)) \phi^1(c_1)\cdots\phi^n(-c_n)  \,dc
\end{align}
and finally integrating on $c$ yields the conclusion.
\end{proof}

 Define 
\begin{align}\notag{}
& \tilde{\phi}^{[n]}(\overline{t}) :=\tilde{\phi}^{[n]}(\overline{t},c) :=\prod_{j=1}^{n}\phi^j(-c_j)\phi^j(-t_j-c_j)
\end{align}

Then
\begin{align}&\label{secexpander}
 \phi^{[n]}(t) = \int \phi^0(t_0+\sum_{j=1}^{n}c_j) \tilde{\phi}^{[n]}(\overline{t})\, dc
\end{align}
and
\begin{align}
\label{alig}
 &\phi^{[n+1]}(t) = \int\phi^{n+1}(-t_{n+1}-c_{n+1})\phi^{n+1}(-c_{n+1}) \phi^0(t_0+\sum_{j=1}^{n+1}c_j) \tilde{\phi}^{[n]}(\overline{t})\, dc\\\notag{}
  =&\int\phi^{n+1}(-t_{n+1}-c_{n+1})\phi^{n+1}(-c_{n+1}) \phi^{[n]}(t_0+c_{n+1};\overline{t}) \, dc_{n+1} 
\end{align}

As a consequence, we have the following.

\begin{lemma}\label{thm:aligs}
For any $n\in\N$ and function bounded function $f:\T^d\times\T^{dn}\ra\C$ set 
\begin{align*}
 &\tri f:\T^d\times\T^{d(n+1)}\ra\C, (x;v)\mapsto f(x;v')\overline{f(x-v_{n+1};v')}
\end{align*}
Then

\begin{align}\label{aligs}
&\int \phi^{n+1}\star f(x-t_0;v'-\overline{t}') \overline{\phi^{n+1}\star f(x-v_{n+1}-t_0;v'-\overline{t}')} \phi^{[n]}(t_0;\overline{t}')\,dt'\\\notag{}=&
\phi^{[n+1]}\ast\tri f (x;v)
\end{align}
\end{lemma}
\begin{proof}
We expand the convolutions on the left side of (\ref{aligs}), obtaining
\begin{align}\label{aligs1}
&\int \phi^{n+1}(-c_{n+1}) f(x+c_{n+1}-t_0;v'-\overline{t}') \\\notag{}\cdot&
\phi^{n+1}(-t_{n+1})\overline{f(x-v_{n+1}+t_{n+1}-t_0;v'-\overline{t}')} \phi^{[n]}(t_0;\overline{t}')\,dc_{n+1}\,dt
\end{align}
Sending $t_{n+1}\mapsto t_{n+1}+c_{n+1}$ and $t_0\mapsto t_0 + c_{n+1}$, this becomes
\begin{align}\label{aligs2}
(\ref{aligs1})=&\int \phi^{n+1}(-c_{n+1})\phi^{n+1}(-t_{n+1}-c_{n+1})\phi^{[n]}(t_0;\overline{t}')\\\notag{}\cdot&
f(x-t_0;v'-\overline{t}') \overline{f(x-v_{n+1}+t_{n+1}-t_0;v'-\overline{t}')} \,dc_{n+1}\,dt
\end{align}
Then applying Fubini's theorem, we have
\begin{align}\label{aligs3}
(\ref{aligs2})=&\int \left[\int \phi^{n+1}(-c_{n+1})\phi^{n+1}(-t_{n+1}-c_{n+1})\phi^{[n]}(t_0;\overline{t}') \,dc_{n+1}\right]\\\notag{}\cdot&
f(x-t_0;v'-\overline{t}') \overline{f(x-v_{n+1}+t_{n+1}-t_0;v'-\overline{t}')} \,dt
\end{align}
By (\ref{alig}), this is
\begin{align}\label{aligs4}
(\ref{aligs3})=&\int \phi^{[n+1]}(t) f(x-t_0;v'-\overline{t}') f(x-t_0-(v_{n+1}-t_{n+1});v'-\overline{t}') \,dt
\end{align}
which is what we sought to show.
\end{proof}

We use the above lemma to show the following, from which Corollary \ref{thm:phin} is derived.

\begin{lemma}\label{thm:phin1}
 Suppose that $\nu_{\iota}, \iota\in\{0,1\}^{j+1}$ are measures on $\T^d\times\T^{dr}$ such that $\tcal{\nu_{\iota}}{\nu_{\iota}}$ exists for each $\iota$. 
 Let $\Phi_{\iota}\in L^1(\T^d)$ for $\iota\in\{0,1\}^{j+1}$. For any $\phi_{\iota}^j\in L^{\infty(\T^d)}$, $j=0,\dots,r$, let $\psi_{\iota}^j =|\phi_{\iota}^j|$ for $0\leq j$, and 
 $\psi_{\iota}^{j+1}=\phi_{\iota}^j$. Then
 \begin{align}\label{mis} &\left|\int \prod_{\iota\in\{0,1\}^{j+1}} \mathcal{C}^{|\iota|}  \Phi_{\iota}\ast\phi_{\iota}^{j+1}\star\phi_{i\iota}^{[r]}\ast\nu_{\iota}(x-iy;u)\,dx\,dy\,du\right|\\\leq&
\left(\prod_{\iota\in\{0,1\}^{j+1}}\|\Phi_{\iota}\|_{L^1}\|\psi_{\iota}^{[j]}\|_{L^1} \right)^{\half}\prod_{i=0}^1\left|\int \prod_{\iota\in\{0,1\}^j}
 |\Phi_{i\iota}|\ast\psi_{i\iota}^{[r+1]}\ast\tcal{\nu_{i\iota}}{\nu_{i\iota}}(x;u,y)\,dx\,dy\,du\right|^{\half}.\label{mit}
 \end{align}
\end{lemma}
\begin{proof}
Changing variables $x\mapsto x_0$, $x_0-y\mapsto x_1$, we have
  \begin{align}&(\ref{mis})^2=\notag{} \left|\int \prod_{i=0}^1\prod_{\iota\in\{0,1\}^{j}} \mathcal{C}^{|\iota|}  \Phi_{i\iota}\ast\phi_{i\iota}^{j+1}\star\psi_{i\iota}^{[r]}\ast\nu_{i\iota}(x_i;u)\,dx_0\,dx_1\,dx_1\,du\right|^2\\=&
  \lim_{n\ra\infty}\left|\int \prod_{i=0}^1\left[\bm{p}(\bm{t})\right]^{\half}\prod_{\iota\in\{0,1\}^{j}} g_{i}(x_i-T_{i\iota,0};u-T_{i\iota,1})\,dx_0\,dx_1\,du\right|^2
  \label{chocafe}  \end{align}
with $t_{i\iota} = (t_{i\iota}^0,t_{i\iota}^1)$, $t_{i\iota}^s=(t_{i\iota,0}^s,t_{i\iota,1}^s), s=0,1, \bm{t_i} =
(t_{i\iota})_{\iota\in\{0,1\}^j}$, $T_{i\iota,s}  = t_{i\iota,s}^{0}+t_{i\iota,s}^{1}$ 
\[\bm{p}(\bm{t})=\prod_{i=0}^1\prod_{\iota\in\{0,1\}^j} \Phi_{i\iota}(t_{i\iota}^0)\phi_{i\iota}^{[r+1]}(t_{i\iota,0}^1;t_{i\iota,1}^1),\]
and
\[g_{i\iota}(x;u) = \mathcal{C}^{|\iota|+i} \phi_{i\iota}\star\phi_n\ast\nu_{i\iota}(x;u).\]

Now we apply Cauchy-Schwarz
\begin{align}\label{mi}
 &(\ref{chocafe})\leq \lim_{n\ra\infty}\prod_{i=0}^1  
 \left(\int \left|\bm{p}(\bm{t})\right|\left[\prod_{s=0}^1\mathcal{C}^s\int \prod_{\iota\in\{0,1\}^{j}} g_{i\iota}(x_i-T_{i\iota,0};u-T_{i\iota,1})\,dx_s\right]\,du\,d\bm{t_i}\right).
\end{align}
Note that by the triangle inequality, we may replace $\bm{p}$ in the above by $\bm{p'}$ which is defined by replacing $\phi_{i\iota}^{[r+1]}$ in the definition of $\bm{p}$ by $\psi_{i\iota}^{[r+1]}$. 
We do so.
Changing variables $x_0\mapsto x$, $x_1\mapsto x - y$, integrating through by $\bm{t_i}$ as well as the now defunct $\bm{t_{i+1(\mod2)}}$ variable, and applying Lemma \ref{thm:aligs}, we have 
\begin{align}\notag{}
 &(\ref{mi})= \left(\prod_{\iota\in\{0,1\}^{j+1}}\|\Phi_{\iota}\|_{L^1} \||\psi_{\iota}|^{[j]}\|_{L^1} \right)\\\cdot&\label{m}
 \lim_{n\ra\infty}\prod_{i=0}^1   \left(
 \int \prod_{\iota\in\{0,1\}^{j}}|\Phi_{i\iota}|\ast\psi_{i\iota}^{[j+1]}\ast\tri(\phi_n\ast\nu_{i\iota})(x_i;u,y)\,dx\,dy\,du\,\right)
\end{align}

and using Lemma \ref{thm:factories} to take evaluate the limit we obtain
\[(\ref{m}) = (\ref{mit})^2.\]

\end{proof}

\begin{proof}[Proof of Corollary \ref{thm:phin}]
The proof follows directly from an induction using Lemma \ref{thm:phin1}.
\end{proof}

Before proving Proposition \ref{thm:GCS}, we must introduce more notation.

We set
\begin{align*}
&\Phi_{\iota_{>j}} :=  \Phi_{n_{\iota_{>j+1}}}^{(j)}\\
& \vec{\bm{n}}' = \left\{ n_{\iota_{>j}}\right\}_{0\leq j\leq k, \iota'\in\{0,1\}^k}\\
&\vec{\bm{n}} = \left\{ n_{\iota_{>j}}\right\}_{0\leq j\leq k+1, \iota\in\{0,1\}^{k+1}}
\end{align*}
and stipulate that $\lim_{\vec{\bm{n}}'\ra\infty}$ refers to each of the $2^k$ limits $\lim_{n_{\iota_{>j}}}$, $j\leq k$, taken in lexicographic order (and similarly for $\lim_{\vec{\bm{n}}\ra\infty}$).

For each $j\leq k+1$ and $\iota_{\leq j}\in\{0,1\}^j$,
we 
set
 \begin{align}\notag{}
  \bm{t^{j}} = \{t^{\iota_{>n}}\}_{1\leq n\leq j, \iota\in\{0,1\}^{k+1}}
 \end{align}
 define the variable $T^{\iota_{\leq j}}$ by

\begin{align}\label{T}
& T^{\iota_1}\equiv 0\\\notag{}&
 T^{\iota_{\leq j}} =  T^{\iota_{\leq j-1}} + t^{(\iota_{>j-1})}_0+\iota_{\leq j}\cdot t^{(\iota_{>j-1})}
 \end{align}
 and define
      
            \begin{align}\label{P}
           & \bm{\Phi^{(0)}} \equiv 1\\\notag{}&
\bm{\Phi^{(j)}}(\bm{t^j}) = \prod_{\iota_{\leq j}\in\{0,1\}^j} \Phi_{\iota_{>j}}(t^{(\iota_{>j})})\bm{\Phi^{(j-1)}}(\bm{t^{j-1}})
\end{align}

 Note that for $i=0,1$ and $j\in\N$
      
      \begin{align}\label{EXc1}
	\tri^j(\bm{\mu_i})=&w^{*}\hyphen\lim_{\vec{\bm{n}}} w^{\ast}\hyphen\lim_{\vec{\bm{m}}\ra\infty}\int \bm{\Phi^{(j)}}(\bm{t^{j}})	
	\prod_{\iota'\in\{0,1\}^{j}}\mathcal{C}^{|{\iota_{\leq j}}|}\Phi_{{\iota\leq j}}\ast\phi_{m_{{\iota'}}}
	\ast\mu_{{\iota_{\iota >j}}}(x-{\iota_{\leq j}}\cdot \bm{u} - T^{{\iota_{\leq j}}})\,d\bm{t}
      \end{align}

\begin{proof}[Proof of Proposition \ref{thm:GCS}]
 
 Since $(\phi_m^{\ast^{k+1}})$ is an approximate identity, by (\ref{EXc1}) we have that 
 
       \begin{align}\notag{}
 \tri^k(\bm{\mu}_{\iota_{k+1}})=&w^{\ast}\hyphen\lim_{\vec{\bm{n}}'\ra\infty} \lim_{\vec{\bm{m}}\ra\infty} \int \bm{\Phi^{(k)}}(\bm{t^{k}})
 \prod_{\iota_{\leq k}\in\{0,1\}^{k}}\mathcal{C}^{|\iota_{\leq k}|}\Phi_{\iota}\ast\phi_{m_{\iota}}^{\ast^{k+1}}\ast\mu_{\iota}(x-\iota\cdot \bm{u} - T^{\iota})\,d\bm{t}
      \end{align}

Thus
      \begin{align*}
  &\left|\langle\bm{\mu}\rangle\right|=\lim_{n\ra\infty} \left|\int \,\overline{\Phi_{n}\ast\tri^k(\bm{\mu}_{1})(x-u_{k+1};\bm{u}')}\,d\tri^k(\bm{\mu}_{0})(x;\bm{u}')\,du_{k+1}\right| \\=
   &\lim_{n_1\ra\infty}\lim_{n_0\ra\infty}\lim_{\vec{\bm{n}}'\ra\infty} \lim_{\vec{\bm{m}}\ra\infty}\left| \int \bm{\Phi^{(k+1)}}(\bm{t^{k+1}})	\prod_{\iota\in\{0,1\}^{k+1}}
   \mathcal{C}^{|\iota|}\Phi_{\iota}\ast\phi_{m_{\iota}}^{\ast^{k+1}}\ast\mu_{\iota}(x-\iota\cdot \bm{u} - T^{\iota})\,d\bm{t}\right|\\\leq&
  \lim_{\vec{\bm{n}}\ra\infty} \lim_{\vec{\bm{m}}\ra\infty}\left| \int \bm{\Phi^{(k+1)}}(\bm{t^{k+1}})	\prod_{\iota\in\{0,1\}^{k+1}}\mathcal{C}^{|\iota|}\Phi_{\iota}\ast\phi_{m_{\iota}}^{\ast^{k+1}}\ast\mu_{\iota}(x-\iota\cdot \bm{u} - T^{\iota})\,d\bm{t}\right|
   \end{align*}

      Set $d\nu_{\iota}(x) = d\mu_{\iota}(x-T^{\iota}-t^{(\iota)})$. 
      
      Now we  apply Corollary \ref{thm:phin} from Section \ref{ch:mollification}, obtaining that this is bounded by
      
                    \begin{align}\label{kthen}
 \limsup_{\vec{\bm{n}}\ra\infty}\lim_{\vec{\bm{m}}\ra\infty}\int  \bm{\Phi^{(k+1)}}(\bm{t^{k+1}})	\Phi_{\iota}(t^{(\iota)})  \prod_{\iota\in\{0,1\}^{k+1}}\left[\phi_{m_{\iota}}^{[k+1]}\ast\tri^{k+1}\nu_{\iota}(\T^{d(k+2)})\right]^{\frac{1}{2^{k+1}}}\,d\bm{t}
      \end{align}
      
      Of course, 
      \begin{align}\notag{}
       \prod_{\iota\in\{0,1\}^{k+1}}\left[\phi_{m_{\iota}}^{[k+1]}\ast\tri^{k+1}\nu_{\iota}(\T^{d(k+2)})\right]^{\frac{1}{2^{k+1}}} = \prod_{\iota\in\{0,1\}^{k+1}}\left[\phi_{m_{\iota}}^{[k+1]}\ast\tri^{k+1}\mu_{\iota}(\T^{d(k+2)})\right]^{\frac{1}{2^{k+1}}}
      \end{align}
since $\nu_{\iota}$ is a shift of $\mu_{\iota}$. And integrating over all the $t$'s in (\ref{kthen}) leaves us with 

\begin{align}\notag{}
 \lim_{\vec{\bm{m}}\ra\infty}\prod_{\iota\in\{0,1\}^{k+1}}\left[\phi_{m_{\iota}}^{[k+1]}\ast\tri^{k+1}\mu_{\iota}(\T^{d(k+2)})\right]^{\frac{1}{2^{k+1}}} = \prod_{\iota\in\{0,1\}^{k+1}}\|\mu\|_{U^{k+1}}
\end{align}
since the $\Phi_{\iota_{>j}}$'s all integrate out to $1$, completing the proof.
\end{proof}

\Ss[Uk is a norm]{$U^k$ is a norm}\label{ch:Norm}

We are now in a position to show $\|\cdot\|_{U^k}$ is indeed a norm. 

 Differing from the usual approach to showing that Gowers norms are norms, we do not need the Gowers-Cauchy-Schwarz inequality to show monotonicity of the $U^k$ norms, or that $\|\mu\|_{U^k}\leq \|\mu\|_{U^{k+1}}$, (and hence positivity of the norm),
 since we have the identity $\|\mu\|_{U^{k}}^{2^{k}}=\sum_{\bm{c}\in\Z^{k}}|\widehat{\tri^{k-1}\mu}(0;\bm{c})|^2\geq|\widehat{\tri^{k-1}\mu}(0;0)|^2=\|\mu\|_{U^{k-1}}^{2^k}$ from Proposition \ref{thm:gowersnorms}. 
But we will use Gowers-Cauchy-Schwarz type inequality to show that the $U^k$ norm satisfies the triangle inequality.

\begin{prop}
 $\|\cdot\|_{U^k}$ defines a norm on the space of measures $\mu$ on $\T^d$ for which $\|\mu\|_{U^k}<\infty$.
\end{prop}

\begin{proof}
Homogeneity of $\|\cdot\|_{U^k}$ is immediate, and that $\|\mu\|_{U^k} = 0$ only if $\mu=0$ follows from Proposition \ref{thm:gowersnorms}. So we are left only to check the triangle inequality.

To do this, let $\mu_1$ and $\mu_2$ be two measures in $U^k$. Then by the definition of the $U^k$ norm and (\ref{EXc1}), 

\begin{align}\label{x4}
 \|\mu_1+\mu_2\|_{U^k}^{2^k} = &\lim_{\vec{\bm{n}}'\ra\infty}
\int \int \bm{\Phi^{(k)}}(\bm{t^{k}})	\prod_{\iota\in\{0,1\}^{k}}\mathcal{C}^{|\iota|}\Phi_{\iota}\ast(\mu_1+\mu_2)(x-\iota\cdot \bm{u} - T^{\iota})\,dt\,dx\,d\bm{u}
 \end{align}

Since convolution is additive, we may write 
\begin{align}\label{x5}
(\ref{x4})=&\int \int \bm{\Phi^{(k)}}(\bm{t^{k}})	\prod_{\iota\in\{0,1\}^{k}}\mathcal{C}^{|\iota|}\Phi_{\iota}\ast(\mu_1+\mu_2)(x-\iota\cdot \bm{u} - T^{\iota})\,dt\,dx\,d\bm{u}\\\notag{}=&
\int \int \bm{\Phi^{(k)}}(\bm{t^{k}})	\prod_{\iota\in\{0,1\}^{k}}\mathcal{C}^{|\iota|}\left[\Phi_{\iota}\ast(\mu_1)(x-\iota\cdot \bm{u} - T^{\iota})+\Phi_{\iota}\ast(\mu_2)(x-\iota\cdot \bm{u} - T^{\iota})\right]\,dt\,dx\,d\bm{u}
 \end{align}
and expanding out the product over $\iota$, this is the same as
\begin{align}\label{x6}
 (\ref{x5})=&\int \int \bm{\Phi^{(k)}}(\bm{t^{k}}) \sum_{\bm{\mu}\in\{\mu_1,\mu_2\}^k} \prod_{\iota\in\{0,1\}^{k}}\mathcal{C}^{|\iota|}\left[\Phi_{\iota}\ast(\mu_{\iota})(x-\iota\cdot \bm{u} - T^{\iota})\right]\,dt\,dx\,d\bm{u}
\end{align}

By Dominated Convergence, we  have
\begin{align}\label{x7}
 (\ref{x6})=&\lim_{\vec{\bm{m}}\ra\infty}\int \int \bm{\Phi^{(k)}}(\bm{t^{k}}) \sum_{\bm{\mu}\in\{\mu_1,\mu_2\}^k} \prod_{\iota\in\{0,1\}^{k}}\mathcal{C}^{|\iota|}\left[\Phi_{\iota}\ast(\phi_{m_{\iota}}^{\ast^k}\ast\mu_{\iota})(x-\iota\cdot \bm{u} - T^{\iota})\right]\,dt\,dx\,d\bm{u}\\\notag{}\notag{}=&
 \lim_{\vec{\bm{m}}\ra\infty}\int \int \bm{\Phi^{(k)}}(\bm{t^{k}}) \sum_{\bm{\mu}\in\{\mu_1,\mu_2\}^k} \prod_{\iota\in\{0,1\}^{k}}\\\notag{}\notag{}&
 \mathcal{C}^{|\iota|}\left[\int \Phi_{\iota}(t^{(\iota)})(\phi_{m_{\iota}}^{\ast^k}\ast\mu_{\iota})(x-\iota\cdot \bm{u} - T^{\iota}-t^{(\iota)})\,dt^{(\iota)}\right]\,dt\,dx\,d\bm{u}
\end{align}
Letting $d\nu_{i}(x) = d\mu_{i}(x-T^{\iota}-t^{(\iota)})$ for $i=1,2$, and applying Fubini's Theorem, this is
\begin{align}\label{x8}
 (\ref{x7})=&\lim_{\vec{\bm{m}}\ra\infty}\int \bm{\Phi^{(k)}}(\bm{t^{k}})\left[\prod_{\iota\in\{0,1\}^k}\Phi_{\iota}(t^{(\iota)})\right] \\\notag{}\notag{}&
 \sum_{\bm{\nu}\in\{\nu_1,\nu_2\}^k} \int  \prod_{\iota\in\{0,1\}^{k}}\mathcal{C}^{|\iota|}(\phi_{m_{\iota}}^{\ast^k}\ast\nu_{\iota})(x-\iota\cdot \bm{u} )\,dx\,d\bm{u}\,d\bm{t}
\end{align}

We use Corollary \ref{thm:phin} from Section \ref{ch:mollification} to obtain the bound 
\begin{align}\label{x9}
 (\ref{x8})\leq&\lim_{\vec{\bm{m}}\ra\infty}\int \bm{\Phi^{(k)}}(\bm{t^{k}})\left[\prod_{\iota\in\{0,1\}^k}\Phi_{\iota}(t^{(\iota)})\right]  \\\notag{}\notag{}&
 \sum_{\bm{\nu}\in\{\nu_1,\nu_2\}^k} \prod_{\iota\in\{0,1\}^k}\mathcal{C}^{|\iota|}\left[\int \phi_{m_{\iota}}^{[k]}\ast\tri^k\nu_{\iota}(x;\bm{u})\,dx\,d\bm{u}\right]^{\frac{1}{2^k}}d\bm{t}
\end{align}

Since 
\begin{align}\notag{}
&\prod_{\iota\in\{0,1\}^k}\mathcal{C}^{|\iota|}\left[\int \phi_{m_{\iota}}^{[k]}\ast\tri^k\mu_{\iota}(x;\bm{u})\,dx\,d\bm{u}\right]^{\frac{1}{2^k}} = \prod_{\iota\in\{0,1\}^k}\left[\int \phi_{m_{\iota}}^{[k]}\ast\tri^k\nu_{\iota}(x;\bm{u})\,dx\,d\bm{u}\right]^{\frac{1}{2^k}}
 \end{align}
 and
 \begin{align}&\notag{}\int \bm{\Phi^{(k)}}(\bm{t^{k}})\left[\prod_{\iota\in\{0,1\}^k}\mathcal{C}^{|\iota|}\Phi_{\iota}(t^{(\iota)})\right]  \,d\bm{t} = 1
\end{align}
we have

\begin{align}\label{x10}
 (\ref{x9})=&\lim_{\vec{\bm{m}}\ra\infty}  \sum_{\bm{\mu}\in\{\mu_1,\mu_2\}^k} \prod_{\iota\in\{0,1\}^k}\left[\int \phi_{m_{\iota}}^{[k]}\ast\tri^k\mu_{\iota}(x;\bm{u})\,dx\,d\bm{u}\right]^{\frac{1}{2^k}}
\end{align}
or
\begin{align}\label{x11}
 (\ref{x10})=&  \sum_{\bm{\mu}\in\{\mu_1,\mu_2\}^k} \prod_{\iota\in\{0,1\}^k}\|\mu_{\iota}\|_{U^k}
\end{align}

This sum is the same as 
\begin{align}\label{x1}
 (\ref{x11})=&  \big[\|\mu_{1}\|_{U^k}+\|\mu_{2}\|_{U^k}\big]^{2^k}
\end{align}

Plugging this back into (\ref{x4}), we have shown that
\begin{align}\label{x2}
\|\mu_1+\mu_2\|_{U^k}\leq&  \|\mu_{1}\|_{U^k}+\|\mu_{2}\|_{U^k}
\end{align}
which is the triangle inequality.

Thus $\|\cdot\|_{U^k}$ is a norm.
\end{proof}

\bibliographystyle{plain}

\bibliography{biblio.bib}

%
%
%

\vskip0.5in

\noindent Marc Carnovale
\\ The Ohio State University \\ 231 W. 18th Ave\\ Columbus Oh, 43210 United States\\ {\em{Email: carnovale.2@osu.edu}}

\end{document}